\documentclass[%draft,%
12pt,oneside,english,a4paper]{amsart}
\usepackage[T1]{fontenc}
\usepackage[utf8]{inputenc}
\usepackage[dvipsnames]{xcolor}
\usepackage[yyyymmdd,hhmmss]{datetime}
\usepackage[numbers,sort,compress]{natbib}

\usepackage{amstext,amsthm,amssymb}
\usepackage{xparse,mathrsfs,mathtools}
\usepackage[english]{babel}
\usepackage[pdftex,unicode=true,%
	pdfusetitle,pdfdisplaydoctitle=true,%
	pdfpagemode=UseOutlines,%
	bookmarks=true,bookmarksnumbered=false,bookmarksopen=true,bookmarksopenlevel=2,%
	breaklinks=true,%
	pdfencoding=unicode,psdextra,
	pdfcreator={},
	pdfborder={0 0 0}
]{hyperref}
\usepackage{lmodern,microtype}
\usepackage{geometry}
\usepackage[capitalise,nameinlink]{cleveref}
\usepackage[cal=zapfc,bb=ams,frak=esstix,scr=boondoxo]{mathalfa}
\usepackage[cal=zapfc,bb=ams,frak=esstix,scr=boondoxo]{mathalfa}

\newtheorem*{thm*}{Theorem}
\newtheorem{thm}{Theorem}[section]
\newtheorem{lem}[thm]{Lemma}

\newtheorem{prop}[thm]{Proposition}
\newtheorem{Def}[thm]{Definition}

\newtheorem{Claim}[thm]{Claim}
\theoremstyle{remark}
\newtheorem{rem}[thm]{Remark}

\newtheorem*{rem*}{Remark}

\NewDocumentCommand\e{ s O{} m }{%
	\IfBooleanTF{#1}{%
		\operatorname{e}_{#2}\parentheses*{#3}%
	}{\operatorname{e}_{#2}\parentheses{#3}}%
}

\DeclarePairedDelimiter\parentheses{\lparen}{\rparen}
\DeclarePairedDelimiter\braces{\lbrace}{\rbrace}

\NewDocumentCommand\set{ s o m o }{%
	\IfBooleanTF{#1}{\IfNoValueTF{#4}{\braces*{#3}}{\braces*{\,#3:#4\,}}}{%
	\IfNoValueTF{#2}{\IfNoValueTF{#4}{\braces{#3}}{\braces{\,#3:#4\,}}}{%
	\IfNoValueTF{#4}{\braces[#2]{#3}}{\braces[#2]{\,#3:#4\,}}}}%
}

\crefformat{equation}{#2(#1)#3}
\crefformat{enumi}{#2(#1)#3}
\numberwithin{equation}{section}
\usepackage{xifthen,ifthen}
\makeatletter
\newcounter{@ToDo}
\newcommand{\todo@helper}[1]{%
	({\color{blue}TODO~\arabic{@ToDo}: {#1\@addpunct{.}}})%
}
\newcommand{\todo}[1]{\stepcounter{@ToDo}%
	\relax\ifmmode\text{\todo@helper{#1}}%
	\else\todo@helper{#1}\fi%
}
\makeatother

\title{On small fractional parts of perturbed polynomials}
\date{\today{}}

\author{Paolo~Minelli}
\address{
	Paolo~Minelli\\%
	Institut für Analysis und Zahlentheorie\\%
	TU~Graz\\%
	Kopernikusgasse~24/II\\%
	8010~Graz\\%
	Austria}
\email{minelli@math.tugraz.at}
\begin{document}
\begin{abstract}
Questions concerning small fractional parts of polynomials and pseudo-polynomials have a long history in analytic number theory. In this paper, we improve on earlier work by Madritsch and Tichy. In particular, let $f=P+\phi$ where $P$ is a polynomial of degree $k$ and $\phi$ is a linear combination of functions of shape $x^c$, $c\not \in \mathbb{N}$, $1<c<k$. We prove that for any given irrational $\xi$ we have
\[\min_{\substack{2\leq p\leq X\\ p \text{ prime}}} \Vert \xi \lfloor f(p)\rfloor\Vert \ll_{f,\epsilon} X^{-\rho(k)+\epsilon},\]
for $P$ belonging to a certain class of polynomials and with $\rho(k)>0$ being an explicitly given rational function in $k$.
\end{abstract}
\maketitle

\section{Introduction and statement of results}
In the present note, a \textit{pseudo-polynomial} is a function $f:\mathbb{R}\to \mathbb{R}$ of the form
\begin{align}\label{pseudopolydef}
f(x)=\sum_{j=1}^d \alpha_jx^{\theta_j}
\end{align}
for $\alpha_j$ positive reals and $1\leq \theta_1<\theta_2<\dots\theta_d$, with at least one non integral $\theta_j$, $1\leq j\leq d$. 
We will write $f=P+\phi$,  where $\phi$ is of the form \cref{pseudopolydef} where \textit{all} the exponents $\theta_j$ are non integral and 
$P$ is the remaining \textit{polynomial part}. We define the degree $\deg(f)$ in the obvious way: as the largest exponent appearing in \cref{pseudopolydef}. For $f=P+\phi$ we say that $f$ is \textit{dominant} if $\deg(f)=\deg(\phi)$, otherwise we will say that $f$ is  \textit{non- dominant}. Bergelson \textit{et.al.} \cite{TichyBergelson} proved, among other results concerning pseudo-polynomials, that  for a given pseudo-polynomial 
$f$, the sequence $(f(p)_p)$ is uniformly distributed modulo 1. Madritsch 
and Tichy \cite{TichyMadr2019} investigated Diophantine properties of pseudo-polynomials along primes. In particular, they proved the following
\begin{thm}[Madritsch-Tichy]\label{MTtheorem}
	Given a pseudo-polynomial $f$, any real $\xi$ and $X\in \mathbb{N}$ sufficiently large, there exists an exponent $\rho_1>0$ such that 
	\begin{align}\label{MTexponent}
	\min_{\substack{2\leq p\leq X\\ p \text{ prime}}}\Vert \xi \lfloor f(p)\rfloor \Vert \ll_f X^{-\rho_1(f)+\epsilon}.
	\end{align}
\end{thm}
Results concerning similar problems (most noticeably polynomials) have a long history, tracing back to Vinogradov (1927) who proved that for any fixed real number $\theta$ and $k\geq 2$ integer one has
\begin{equation}\label{thm:Vinogradov:originalresult}
\min_{1\leq n\leq X} \Vert n^k \theta\Vert \ll X^{-\eta_1(k)+\epsilon},
\end{equation}
for $\eta_1(k)>0$ and with the implied constant depending only on $k$ and 
$\epsilon$. Another neighboring problem was posed by Davenport (1967): for $f$ a polynomial of degree $k$ with $f(0)=0$ and at least one irrational coefficient, establish
\begin{equation}\label{prob:Davenport}
\min_{1\leq n\leq X} \Vert f(n)\Vert\ll X^{-\eta_2(k)+\epsilon},
\end{equation}
with $\eta_2(k)>0$ (conjecturally $\eta_2(k)=1$) and the implicit constant depending only upon $k$ and $\epsilon$.
For \cref{thm:Vinogradov:originalresult} and $k=2$, the best known exponent is $\rho_1(2)=4/7$ provided by Zaharescu \cite{Zaharescu1995}. For 
$k\geq 11$, the best known exponent for \cref{thm:Vinogradov:originalresult} is given by Wooley \cite{Wooley1993}. We refer to Baker \cite{BakerIneq} \cite{Baker2016} for the current records concerning \cref{thm:Vinogradov:originalresult}. Concerning \cref{prob:Davenport}, the optimal exponents are due to Baker \cite{Baker2016}, \cite{Baker1982}: $\eta_2(k)=1/2k(k-1)$ (for $k\geq 8$) resp. $\eta_2(k)=2^{1-k}$ (for $2\leq k\leq 7$). Similar results were obtained for the special case of primes, see Harman \cite{HarmanTrigI} and the more recent work of Baker \cite{Bakerprimes2017}, \cite{Bakerprimes2018}. 
\subsection{Goal of the present note}
In an earlier work (see \cite{minelli2021small}) we improved the results in \cite{TichyMadr2019} for the case of dominant $f$. More specifically, we obtained

\begin{thm}\label{Ourtheorem}
	Let $f$ be a dominant pseudo-polynomial of degree $\theta>3$ and let  $\xi$ be a real number. Then, we have
	\begin{align}\label{ourexponent}
	\min_{\substack{2\leq p\leq X\\ p \text{ prime}}}\Vert \xi \lfloor f(p)\rfloor \Vert \ll_f X^{-\rho(\theta)+\epsilon},
	\end{align}
	where
	\begin{align*}
	\rho(f)=\frac{1}{3}\frac{1}{8\theta^2+12\theta+10}.
	\end{align*}
\end{thm}
The goal of the present paper is to complement the above result, improving on \cref{MTtheorem} in the non-dominant case subjected to a certain condition for the polynomial part. In the case under consideration, the pseudo part may be regarded as a particular perturbation of the polynomial. 
\\\
\\\
We remark that for technical reasons, the case under consideration requires more work than in \cite{minelli2021small}. In particular, the estimation of Type I and Type II sums we will require in this work are slightly more intricate and need a further assumption on the structure of the polynomial. This motivates the following definition.
\subsection{Statement of results} 
\begin{Def}[Property (F)]\label{def:propertyF:nondominant:prime}
 Let $P(x)=\sum_{j=1}^k \alpha_j x^j$ be a polynomial. We call $P$ \textbf{full} if $\alpha_j\not =0$ for all $1\leq j\leq k$. We say that a pseudo polynomial $f=P+\phi$ has property (F) if
 \begin{enumerate}
 \item $k:=\deg(P)\geq 12.$
 \item $\theta:=\deg(\phi)>4.$
 \item $k>\theta.$
 \item $P$ is full.
 \end{enumerate}
\end{Def}

\begin{Def}(Three technical parameters)
	Let $f=P+\phi$ be a pseudo-polynomial with property (F), with $\deg(P)=k$ and $\deg(\phi)=\theta$. We define the following quantities:
	\begin{align*}
	\tau_1:=\frac{1}{k(k-1)},
	\end{align*}
	\begin{align*}
	\rho_\star:=\min\left(\frac{\theta-1}{k}-\tau_1, \frac{1}{4}-\tau_1\right)
	\end{align*}
	and	 
\begin{align}\label{rhod:definition:maintheorem:nondominant:prime}
	\rho(f):=\frac{1}{3}\min\left(\frac{2}{(k+1)(\frac{3}{2}k+\frac{5}{2})^2}, \frac{\rho_\star}{k(k-1)^2}\right)
\end{align}
\end{Def}

\begin{rem}
	In the definition above we dropped explicit dependency in $f$ for the parameters $\tau_1$ and $\rho_\star$. This will cause no confusion in the sequel, as there will be no ambiguity concerning the function $f$ under consideration.
\end{rem}

With this definitions in mind we can state the main result of the present note
\begin{thm}\label{theorem:maintheoremnondominat:prime}
Let $f=P+\phi$ have property (F) with $P$ of degree $k$ and $\phi$ of degree $\theta$. Let $\xi$ be a real number. Then
\begin{align*}
\min_{2\leq p\leq X} \Vert \xi \lfloor f(p)\rfloor\Vert \ll_f X^{-\rho(f)+\epsilon}.
\end{align*}
\end{thm}

\subsection*{Comparison with former results}
For a polynomial with property (F), our result improves on the exponent obtained by Madritsch and Tichy. In particular, the exponent \cref{rhod:definition:maintheorem:nondominant:prime} is lower bounded by the reciprocal of quartic polynomial in the degree of $f$, while the exponent obtained in \cite{TichyMadr2019} decays exponentially in the degree of $f$.
\subsection*{Structure of the paper:}
This paper is structured as follows. In \cref{prelimresults} we state the 
required exponential sums bounds as well as some standard tools we will require later. In \cref{theproof} we use these results 
to prove the theorem. The proof of the exponential sums bounds is the most technical part, and is deferred to \cref{Section:typeI},  \cref{section:typeII} and \cref{Section:diophantine:claims}. The strategy of proof is similar to the one employed in \cite{TichyMadr2019} and \cite{minelli2021small}. Most of our efforts will be spent in the last two sections for obtaining bounds for prime exponential sums.
\subsection*{Notation:}
We will say that a sequence of complex numbers $(a_n)_n$ is divisor bounded if $\vert a_n\vert \leq d_4(n)$, where $d_4(n):=\sum_{m_1m_2m_3m_4=n} 1$. With $a\sim A$ we will mean $\frac{A}{2}<a\leq A$. The Vinogradov symbol $\ll$ has the usual meaning. In \cref{Section:typeI} and \cref{section:typeII} we will make use of Weyl differencing. Depending on the situation, the differencing process will take place on different variables. For $f$ a function of one real variable we will write
\begin{align*}
f_h(u,v):=f(u(v+h))-f(uv),
\end{align*}
the second variable indicating where the differencing process takes place.
\subsection*{Acknowledgements}
We thank Daniel El-Baz and Christoph Aistleitner for comments on an early version of the paper. We are also indebted with the anonymous reviewer for precise and pertinent comments leading to an improvement of our work and for polishing our English sentences. The author was supported by the Austrian Science Fund (FWF) project I-3466.

\section{Preliminary results and lemmas}\label{prelimresults}
In the next sections we will make use (repeatedly) of the following results. The following lemma is classical
\begin{lem}[See e.g Theorem 2.2 in \cite{BakerIneq}]\label{lemma:thelargesieveineq}
Let $M$ be a positive integer and let $(x_j)_j$ be a sequence of $N$ elements. Suppose that $\Vert x_j\Vert\geq \frac{1}{M}$ for all $1\leq j\leq N$. Then there is an integer $m$ in the range $1\leq m\leq M$ such that
\begin{align*}
\Big\vert \sum_{j=1}^N e\left(mx_j\right)\Big\vert \geq \frac{N}{6M}.
\end{align*}
\end{lem}
When estimating exponential sums at various frequencies we will invoke the following two (interconnected) results. 
\begin{lem}[See Theorem 5 in \cite{BourgainVinogradov}]\label{lemma:bourgainvinogradov:nondominant:prime}
Let $P(x)=\sum_{j=1}^k \alpha_jx^j$ be a polynomial of degree $k\geq 3$, then we have
\begin{align*}
\sum_{n\leq X} e\left(P(n)\right)\ll X^{1+\epsilon}\left(\frac{1}{q}+\frac{1}{X}+ \frac{q}{X^{j}}\right)^{\frac{1}{k(k-1)}},
\end{align*}
provided there an index $2\leq j\leq k$ and a pair of coprime integers $(a,q)$ with
\begin{align*}
\vert \alpha_j-\frac{a}{q}\vert\leq \frac{1}{q^2}.
\end{align*}
\end{lem}

\begin{lem}[Heath-Brown's derivative test, see e.g \cite{HBKDT}, \cite{KumchevPetrov2019}]\label{HBlemma}
Let $F$ and $X$ be large parameters and assume $X\leq Y\leq 2X$. Let $k\geq 3$ be an integer, and $f: [X, Y]\to \mathbb{R}$ be a $k$-times continuously derivable function which satisfies the following 
\begin{align}\label{HBlemmaeq1}
FX^{-k}\ll \Big\vert f^{(k)}(x)\Big\vert \ll FX^{-k} \qquad x\in (X,Y],
\end{align}
Then we have the estimate
\begin{equation}\label{Hblemmaeq2}
\sum_{X <n\leq Y} e(f(n)) \ll X^{1+\epsilon} \times \left[\left(FX^{-k}\right)^\frac{1}{k(k-1)}+X^{-\frac{1}{k(k-1)}}+F^{-\frac{2}{k^2(k-1)}}\right],
\end{equation}
where the implicit constant above may depend upon those in (\ref{HBlemmaeq1}) and the level of differentiation $k$.
\end{lem}

\begin{lem}[Heath-Brown's decomposition, see e.g \cite{HBShapiro}]\label{lem:heathbrowdecomposition}Let $3\leq V<Z<X$ and suppose that $z$ is an half integer. Assume further that these variables satisfy $Z\geq 4U^2$, $X\geq 64Z^2U$, $V^3\geq 32X$. Let now $f$ be a function supported in $[X/2,X]$ and bounded by $f_0$ in this interval. Then we have
\begin{align*}
\sum_{n\sim X} \Lambda(n)f(n)\ll f_0+K\log X+ L \log^8 X,
\end{align*}
where 
\begin{align*}
K:=\max_{N}\sum_{m=1}^\infty d_3(m)\Big\vert \sum_{\substack{Z<n<N}} f(mn) \Big\vert,
\end{align*}
and
\begin{align*}
L:=\sup_{g}\sum_{m=1}^\infty d_4(m)\Big\vert \sum_{U<n\leq V} g(n)f(mn)\Big\vert,
\end{align*}
where the supremum is taken over all arithmetic functions that are bounded by $d_3(n)$ (i.e $\vert g(n)\vert \leq d_3(n)$).
\end{lem}
In order to prove the theorem, we need some information concerning prime exponential sums. 
\begin{lem}[Prime exponential sums]\label{lem:primesums}
Let $f=P+\phi$ be a pseudo-polynomial with property (F), where $P$ has degree $k$, and  $\deg(\phi)=\theta$. Let $X^{-\frac{2k}{3}}\ll y\ll X^{1/3}$. Then 
\begin{align*}
\sum_{p\leq X} e\left(y f(p)\right)\ll_f X^{1-3\rho+\epsilon}.
\end{align*}
\end{lem}

The proof of the latter lemma follows \textit{mutatis mutandis} through the same standard lines of \cite{minelli2021small}. Nevertheless, we shall give the details for completeness.
\begin{proof}
By partial summation, a bound for the sum
\begin{align}\label{lemma:primesums:vonmangoldslongrange:nondominant:prime}
\sum_{n\leq X} \Lambda(n) e\left(yf(n)\right)
\end{align}
can be turned (up to a logarithmic factor due to partial summation and an 
additive error of $O(X^{\frac{1}{2}})$ due to higher prime powers) into a 
bound for our sum. Split now the sum (\ref{lemma:primesums:vonmangoldslongrange:nondominant:prime}) into dyadic ranges $n\sim X_i$ with $X_i=\frac{X}{2^i}$ and $0\leq i\ll \log X$. 
Now, we estimate trivially in all the intervals of size $\ll X^{1-\rho}$. 
For the remaining intervals (which we denote by $Y$ for ease of notation), an application of \cref{lem:heathbrowdecomposition} with parameters $U=c_1Y^{1/5}$ and $V=c_2Y^{\frac{1}{3}}$ and $Z$ being the half integer closest to $c_3X^{2/5}$ (with constants selected in order to fit the hypotheses of \cref{lem:heathbrowdecomposition})  leads to
\begin{align}\label{equation:proof:profexpsumsprimes:sumafterhbid:nondominant:prime}
\sum_{n\sim Y} \Lambda(n)e\left(yf(mn)\right)\ll f_0+ K\log X+L\log^2X.
\end{align}
Now the sum
\begin{align*}
K=\sum_{m=1}^\infty d_3(n)\Big\vert \sum_{\substack{Z<n\leq Y\\ nm\sim Y}}e\left(yf(mn)\right)\Big\vert=\sum_{m=1}^\infty \sum_{\substack{Z<n\leq Y}} a_m e\left(yf(mn)\right),
\end{align*}
with $(a_m)_m\subset \mathbb{C}$, can be further decomposed (dyadic split 
on the inner range for $n$) into $\ll \log X$ sub sums of shape
\begin{align}\label{equation1:proofoflemmaexpsumsamongprimes:nondominant:prime}
\sum_{m=1}^M\sum_{\substack{n\sim N\\ mn\sim Y}} a_m e\left(yf(mn)\right),
\end{align}
where $M\leq Y/Z\ll Y^{3/5}$, $\vert a_m\vert \leq d_3(m)$ and $(a_m)_m\subset \mathbb{C}$. We can proceed similarly for the sum 
$L$, decomposing this into $\log^2 X$ sub sums of shape
\begin{align}\label{equation2:proofoflemmaexpsums:nondominant:prime}
\sum_{m \sim M}\sum_{\substack{n\sim N\\ mn \sim Y}} a_m b_n e\left(yf(mn)\right),
\end{align}
where $Y^{1/5}\ll U<M<V\ll Y^{\frac{1}{3}}$ and $(a_m)_m$ and $(b_n)_n$ sequences of divisor bounded numbers. Now, we see that the sums \cref{equation1:proofoflemmaexpsumsamongprimes:nondominant:prime} 
and \cref{equation2:proofoflemmaexpsums:nondominant:prime} above have almost identical shape as the sums in \cref{prop:type1expsums:nondominant:prime} and \cref{prop:type2expsums:nondominant:prime}. The only difference is that the interval for $y$ in the statements of \cref{prop:type1expsums:nondominant:prime} and \cref{prop:type2expsums:nondominant:prime} are given with respect to the summation range $mn\sim X$, while in \cref{equation1:proofoflemmaexpsumsamongprimes:nondominant:prime} and \cref{equation2:proofoflemmaexpsums:nondominant:prime} the range for $y$ is still with respect to $X$ but the summation range has size $mn\sim Y$. However, since 
we have $Y\gg X^{1-\rho}$ and we are working in the range $X^{-\frac{2k}{3}}\ll y\ll X^{1/3}$, writing the constraints for $y$ in terms of $Y$, 
we can apply \cref{prop:type1expsums:nondominant:prime} and \cref{prop:type2expsums:nondominant:prime}. Gathering everything together and using (\ref{equation:proof:profexpsumsprimes:sumafterhbid:nondominant:prime}) we conclude
\begin{align*}
\sum_{n\sim Y} \Lambda(n)e\left(yf(mn)\right)\ll X^{1-\rho+\epsilon}.
\end{align*}
\end{proof}

\section{Proof of the theorem}\label{theproof}
\begin{proof}
We proceed by contradiction. Assume that   
\begin{equation}\label{contradeq}
\min_{1\leq p\leq X} \Vert \xi \lfloor f(p)\rfloor\Vert\geq X^{-\tilde{\rho}}
\end{equation}
for some small exponent $\tilde{\rho}$ with $\rho(f)>\tilde{\rho}>0$. We set
\begin{align*}
M:=\lfloor X^{\tilde{\rho}}\rfloor.
\end{align*}
Now, by \cref{lemma:thelargesieveineq}, there is a $m\leq M$ such that 
\begin{align}\label{eq:proofofthetheorem:expsumslowerbound:nondominant:prime}
\sum_{p\leq X} e\left(m\xi  \lfloor f(p)\rfloor\right)\gg X^{1-\tilde{\rho}}.
\end{align}
\subsection{Case I}($\xi$ not too well approximable)
\\\
Assume $\Vert m\xi\Vert\geq X^{-\frac{2k}{3}}$. Proceeding as in \cite{TichyMadr2019}, the problem reduces to the estimation of the following three 
sums:

\begin{enumerate}
	\item \(\displaystyle \frac{1}{q}\Big\vert \sum_{p\leq X} e\left(m\xi f(p)\right)\Big\vert\),
	\item \(\displaystyle \sum_{0\leq \vert h\vert\leq H} \frac{1}{h} \Big \vert \sum_{p\leq X} e\left((m\xi+h)f(p)\right)\Big\vert\),
	\item \(\displaystyle \frac{1}{H+1}\sum_{\vert h\vert \leq H} \left(1-\frac{\vert h\vert}{H+1}\right)\Big\vert \sum_{p\leq X} e\left(hf(p)\right)\Big\vert, \)
\end{enumerate}

where we selected $H=X^{1/6}$ and $q$ is a parameter to be specified later. As we are working under the condition $\Vert m\xi\Vert\geq X^{-\frac{2k}{3}}$ we see $\vert m\xi+h\vert \gg X^{-\frac{2k}{3}}$ for all $0\leq h\leq H$. Hence, we may apply  \cref{lem:primesums}. This gives us 
\begin{align*}
\sum_{p\leq X} e\left(m\xi\lfloor f(p)\rfloor\right)&\ll_f qX^{1-\rho+\epsilon} \left( \frac{1}{q}+ \sum_{0<\vert h\vert \leq H} \frac{1}{\vert h \vert }+\frac{1}{H+1}\sum_{\vert h\vert \leq H} \left(1-\frac{h}{H+1}\right)\right)+ O\left(\frac{Xm}{q}\right)\\
&\ll_f X^{1-\rho+\epsilon}+qX^{1-\rho+\epsilon}+\frac{Xm}{q}.
\end{align*}
Selecting $q=\lfloor m^{\frac{1}{2}}X^{\frac{\rho}{2}}\rfloor$ we obtain the bound
\begin{align*}
\ll N^{1-\frac{\rho}{2}+\epsilon}m^{\frac{1}{2}}\ll N^{1-\frac{\rho}{2}+\frac{\tilde{\rho}}{2}+\epsilon},
\end{align*}
hence a contradiction to the lower bound (\ref{eq:proofofthetheorem:expsumslowerbound:nondominant:prime}) since
\begin{align}\label{condition1}
\tilde{\rho}<\rho(f)\leq \frac{\rho}{3}.
\end{align}
\subsection{Case II}($\xi$ is too well approximable)
\\\
If $m=1$ then taking $p=2$ we get 
\begin{align*}
\min_{2\leq p\leq X}\Vert \xi \lfloor f(p)\rfloor\Vert\ll 2^k X^{-\frac{2k}{3}}\ll X^{-\rho(f)}, 
\end{align*}
which contradicts \cref{contradeq}.
If $m\geq 2$, then \cref{claim:xitoowel:multiple:nondominant:prime} below 
shows that there is a prime $p\ll X^{1/3+\epsilon}$ such that $\lfloor f(p)\rfloor$ is a multiple of $m$. Hence, for this $p$
\begin{align*}
X^{-\rho(f)}\leq \Vert \xi \lfloor f(p)\rfloor\Vert \ll \Vert \xi m\Vert \frac{\lfloor f(p)\rfloor}{m}\ll X^{-\frac{2k}{3}}\frac{X^{\frac{k}{3}+k \epsilon}}{m}\ll X^{-\frac{k}{3}+k\epsilon},
\end{align*}
which is impossible as $k\geq 5$. The following claim concludes the proof.
\begin{Claim}\label{claim:xitoowel:multiple:nondominant:prime}
 Let $2\leq m\leq M$. Then, for sufficiently large $X$ there is a prime $p\ll X^{\frac{1}{3}+\epsilon}$ for $\epsilon>0$, such that $\lfloor f(p)\rfloor$ is divisible by $m$. 
\end{Claim}
\begin{proof}[Proof of \cref{claim:xitoowel:multiple:nondominant:prime} ]
	Follows \textit{mutatis mutandis} from Lemma 2.6 in \cite{minelli2021small}
\end{proof}
\end{proof}

\section{Type I estimates}\label{Section:typeI}
The goal of the present section is to establish the following proposition.
\begin{prop}[Type I]\label{prop:type1expsums:nondominant:prime}
	Let $f=P+\phi$ have property (F) with $P$ having degree $k$ and $\deg(\phi)=\theta$. Let then $X^{-\frac{23k}{30}}\ll y\ll X^{1/2}$ and let 
	$M\leq X^{3/5}$. Finally let $(a_m)_m\subset \mathbb{C}$ be a divisor bounded sequence. Then
	\begin{align*}
	\sum_{m\leq M}\sum_{nm\sim X} a_m e\left(y f(mn)\right) \ll_f X^{1-\frac{2\rho_\star}{k(k-1)^2}+\epsilon}.
	\end{align*}
\end{prop}

\begin{rem}
	In order to proof \cref{prop:type1expsums:nondominant:prime} we must distinguish among "high frequencies" and "low frequencies". The first case corresponds to the range where the derivative test returns an appropriate bound. The second range corresponds to the range where we can discard the perturbation. It is worth to point out that in case $\theta$ is relatively large compared to $k$, say $\theta$ a bit larger than $\frac{23}{30}k$, then \cref{prop:typeI:intfreq} is superfluous, as the full range is already covered by \cref{prop:highfrequencies:typeI:nondominant:prime}.
\end{rem}

\begin{prop}[High frequencies]\label{prop:highfrequencies:typeI:nondominant:prime}
Let $f=P+\phi$ be a pseudo-polynomial with property (F), with $P$ of degree $k$ and $\deg(\phi)=\theta$. Let then $X^{-\theta+\rho^\star}\ll y\ll X^{1/2}$ and let $M\leq X^{3/5}$. Finally, let $(a_m)_m\subset \mathbb{C}$ be a divisor bounded sequence. Then
\begin{align}\label{prop:highfrequenciesbound:nondominant:prime}
\sum_{m\leq M}\sum_{nm\sim X} a_m e\left(y f(mn)\right) \ll_f X^{1-\frac{2\rho_\star}{k(k-1)^2}+\epsilon}.
\end{align}
\end{prop}

\begin{proof}
In this range we apply \cref{HBlemma} with an appropriate selection of the differentiation parameter. Let $X^\alpha=yX^\theta$ we will write $X_m$ for $X/m$. We select $j:=\max\left(k+1,\lceil \frac{5}{2}\alpha\rceil+2\right)$ and apply \cref{HBlemma} with $F=yX^\theta=X^\alpha$:
\begin{align}\label{equation:type1:highfreq:equation1}
\sum_{m\leq M}\sum_{n\sim X_m} a_m e\left(yf(mn)\right)&\ll_f \sum_{m\leq 
M} \vert a_m\vert \Big\vert \sum_{n\sim X_m} e\left(yf(mn)\right)\Big\vert\\ \nonumber
&\ll_f \sum_{m\leq M} \left( X^{\frac{\alpha}{j(j-1)}} X_m^{1-\frac{j}{j(j-1)}+\epsilon}+X_m^{1-\frac{1}{j(j-1)}+\epsilon}+ X_m^{1+\epsilon}X^{-\frac{2\alpha}{j^2(j-1)}}\right)\\ \nonumber
&\ll_f X^{1+\frac{\alpha-\frac{2}{5}j}{j(j-1)}+\epsilon}+ X^{1-\frac{2}{5j(j-1)}+\epsilon}+X^{1-\frac{2\alpha}{j^2(j-1)}+\epsilon}.
\end{align}
Now, if $j=k+1$ (notice that this eventually takes place when $y$ is about its lower bound in the statement) then we have $\alpha\leq \frac{2}{5}(k-1)$. Hence
\begin{align*}
\frac{\alpha-\frac{2}{5}k}{k(k+1)}\leq -\frac{2}{5k(k+1)}.
\end{align*}
Thus
\begin{align}\label{Proofprop:highfreq:intermediateequation:nondom:prime}\nonumber
\sum_{m\leq M}\sum_{n\sim X_m} a_m e\left(yf(mn)\right) &\ll_f X^{1-\frac{2}{5k(k+1)}+\epsilon}+X^{1-\frac{2}{5k(k+1)}+\epsilon}+X^{1-\frac{2\rho_\star}{k(k+1)^2}+\epsilon}\ll_f X^{1-\frac{2\rho_\star}{k(k-1)^2}+\epsilon}.
\end{align}
We turn now to the case $j>k+1$. Notice that since $\theta<k$ and $y\ll X^{1/2}$, we have $k+\frac{1}{2}>\alpha$, hence $\lceil \frac{5}{2}\alpha\rceil +2=j\leq \frac{5}{2}k+\frac{17}{4}$. Hence, for the exponent of the first term we have
\begin{align*}
\frac{\alpha-(\lceil \frac{2}{5}\alpha\rceil +2)\frac{2}{5}}{(\lceil \frac{5}{2}\alpha\rceil+2)(\lceil \frac{5}{2}\alpha\rceil+1)}&\leq -\frac{4}{5}\frac{1}{(\lceil \frac{5}{2}\alpha\rceil +2)(\lceil \frac{5}{2}\alpha\rceil+1)}+\frac{\alpha-\frac{2}{5}\lceil \frac{5}{2}\alpha\rceil}{(\lceil \frac{5}{2}\alpha \rceil +2)\lceil \frac{5}{2}\alpha\rceil}\\
&\leq -\frac{4}{5(\lceil \frac{5}{2}\alpha\rceil +2)(\lceil \frac{5}{2}\alpha\rceil+1)}\\
&\leq -\frac{4}{5} \frac{1}{(\frac{5}{2}k+\frac{17}{4})(\frac{5}{2}k+\frac{13}{4})}\\
&\leq -\frac{1}{12k^2},
\end{align*}
where the last inequality is valid for $k>9$. So the first term in \cref{equation:type1:highfreq:equation1} is at most 
\begin{align*}
\ll_f X^{1-\frac{1}{12k^2}+\epsilon}.
\end{align*}
Proceeding in the same was we see that the second term is bounded by
\begin{align*}
\ll_f X^{1-\frac{1}{20k^2}+\epsilon}.
\end{align*}
For the third term in \cref{equation:type1:highfreq:equation1}, we notice that in order to have $j>k+1$ we need $\alpha>2$ (recall $k>11$ by assumption). Now we compute
\begin{align*}
\frac{2\alpha}{j^2(j-1)}= \frac{2\alpha}{\left(\frac{5}{2}\alpha+3\right)^2\left(\frac{5}{2}\alpha+2\right)}\geq \frac{2k+1}{\left(\frac{5}{2}k+\frac{17}{4}\right)^2\left(\frac{5}{2}k+\frac{13}{4}\right)}\geq \frac{1}{12k^2}
\end{align*}
as the central term is a decreasing function for $\alpha>2$ and $\alpha\leq \theta +\frac{1}{2}<k+\frac{1}{2}$. We conclude that if $j\not=k+1$, then 
(\ref{equation:type1:highfreq:equation1}) is at most
\begin{align*}
&\ll_f X^{1-\frac{1}{20k^2}+\epsilon}.
\end{align*}
Noticing now that $\rho_\star<1/4$, for $k\geq 12$ we see that 
\begin{align*}
\frac{2\rho_\star}{k(k-1)^2}\leq \frac{1}{20k^2},
\end{align*}
from which we deduce the stated bound.
\end{proof}

\begin{prop}[Intermediate and small frequencies]\label{prop:typeI:intfreq}
Let $f=P+\phi$ be a pseudo-polynomial with property (F) where $P$ has degree $k$ and $\deg(\phi)=\theta$. Let then $X^{-\frac{23k}{30}}\ll y\ll X^{-\theta+\rho_\star}$ and let $M\leq X^{3/5}$. Finally let $(a_m)_m\subset \mathbb{C}$ denote a divisor bounded sequence. Then
\begin{equation}\label{prop:intermediatefrequencies:type1:nondominant:prime}
\sum_{m\sim M}\sum_{\substack{n\sim N\\nm\sim X}} a_m e\left(y f(mn)\right) \ll_f X^{1-\frac{1}{5k(k-1)}+\epsilon}.
\end{equation}
\end{prop}

\begin{proof}
Let us denote the left hand side of \cref{prop:intermediatefrequencies:type1:nondominant:prime} by $S$. Set 
\begin{align}\label{equation:csplit:propintermfreq:type1:nondominant:prime}
c:=\min\left(\frac{\theta-1}{k}, \frac{1}{2}+\rho_\star \right)
   \end{align}
and let us assume first that $M\ll X^{c}$. In this case, an application of Cauchy's inequality gives us
\begin{align*}
S^2&\ll \sum_{m\sim M} a_m^2 \times \sum_{m\sim M}\Big\vert \sum_{nm\sim X} e\left(yf(mn)\right)\Big\vert^2.
\end{align*}
Now we apply the Weyl-Van der Corput inequality with $H=X^{\tau}$ (to be specified later, in particular we will select $\tau<\tau_1$) on the innermost sum. We obtain
\begin{align*}
\Big\vert \sum_{n\sim N} e\left(yf(mn)\right)\Big\vert^2\ll \frac{N^2}{H}\log^3 X+ \frac{N}{H}\sum_{1\leq \vert h\vert \leq H} \left(1-\frac{\vert h\vert}{H+1}\right)\sum_{\substack{n\sim N\\ mn\sim X\\(n+h)m\sim X}} e\left(yf_h(m,n)\right).
\end{align*}
 Notice that at the price of a negligible additive error of magnitude $NH$, we can drop the third condition on the summation in the innermost sum. Now we show, we can neglect the noise $y\left(\phi(m(n+h))-\phi(mn)\right)$ as it can be removed by partial summation. Indeed:
\begin{align*}
\sum_{\substack{n\sim X_m\\ mn\sim X}} e\left(yf_h(m,n)\right)
&\ll \max_{\frac{X_m}{2}\leq Y\leq X_m}\Big\vert \sum_{\substack{\frac{X_m}{2}\leq n\leq Y}} e\left(y P_h (m,n)\right) \Big\vert \times \left(1+y \int_{\frac{X_{m}}{2}}^{X_m}\phi'(m(t+h))-\phi'(mt) dt \right).
\end{align*}
The integral above is bounded by 
\begin{align*}
y\frac{X}{M}M^{\theta}N^{\theta-2}H\ll X^{-\theta+\rho_\star} \frac{X^\theta}{N}\ll \frac{X^{\rho_\star+\tau}}{N} \ll X^{\rho_\star+\tau-1+c}\ll 1,
\end{align*}
since the exponent is at most $-\frac{1}{2}+2\rho_\star+\tau$, which is negative, as a consequence of our selection of $\tau$ and the definition of $\rho_\star$. Hence we only need to estimate
\begin{align*}
\max_{X_m/2\leq Y\leq X_m} \Big\vert \sum_{\substack{X_m\leq n\leq Y}}e\left(y P_h(m,n)\right) \Big\vert,
\end{align*}
for which it suffices, by triangle inequality, to bound
\begin{align*}
\sum_{\substack{n\sim X_m\\ mn\sim X}} e\left(yP_h(m,n)\right).
\end{align*}
Now, notice that $yP_h(m,n)$, when regarded as a function in $n$, is a polynomial of degree $k-1$ in $n$ with \textbf{leading} coefficient $ya_k(km^kh)$ (recall differencing takes place on the second variable). Now we prove that this leading coefficient can be appropriately approximated by a rational. For sake of clarity, we deferred the proof of this claim to \cref{Section:diophantine:claims}.

\begin{Claim}\label{Claim:Dioph} Let $m\sim M$, with $M\ll X^c$ and $c$ defined as in \cref{equation:csplit:propintermfreq:type1:nondominant:prime}. Let $X^{-\frac{23k}{30}}\ll y\ll  X^{-\theta+\rho_\star}$. Then, there is a pair of coprime integers $(a,q)$ with 
\begin{align}\label{Dioph:prop:1}
\Big\vert  qya_k(km^kh)-a\Big\vert \leq X_m^{2-k},
\end{align}
and
\begin{align}\label{Dioph:prop:2}
X_m\ll q\ll X_m^{k-2}.
\end{align}
\end{Claim}
We can now apply \cref{lemma:bourgainvinogradov:nondominant:prime} with the pair $(a,q)$ provided by \cref{Claim:Dioph}. Doing this we obtain
\begin{align*}
\sum_{n\sim X_m} e\left(yP_h(m,n)\right)\ll_f X_m^{1-\frac{1}{(k-1)(k-2)}+\epsilon}.
\end{align*}
Whence, summing over $m$
\begin{align*}
S^2&\ll_f \frac{X^{2+\epsilon}}{H}+X^{1+\epsilon}\sum_{m \sim M} X_m^{1-\frac{1}{(k-1)(k-2)}+\epsilon}\\
&\ll_f \frac{X^{2+\epsilon}}{H}+ X^{2-\frac{2}{5(k-1(k-2))}+\epsilon},
\end{align*}
where in the last step we used the fact $M\ll X^{\frac{3}{5}}$. Selecting now $\tau=\frac{2}{5(k-1)(k-2)}$ we conclude that 
\begin{align}\label{prop:intermediatefrequencies:type1:firsthalfbound:nondominant:prime}
S\ll_f X^{1-\frac{1}{5(k-1)(k-2)}+\epsilon}.
\end{align}

We move now to the \textbf{case $M\gg X^c$}. In this case we exchange the roles of the variables:
\begin{align*}
S=\sum_{\substack{m\sim M\\ n\sim N\\nm\sim X}} a_m e\left(yf(mn)\right).
\end{align*}
By applying Cauchy's inequality we obtain
\begin{align*}
S^2\ll \frac{X}{M}\times \sum_{n\sim \frac{X}{M}}\Big\vert \sum_{m\sim \frac{X}{n}} a_m e\left(yf(mn)\right)\Big\vert^2.
\end{align*}
Estimating the innermost term by Weyl-Van der Corput inequality with parameter $H=X^{\tau}$ to innermost sum (we will choose $\tau$ later) we obtain
\begin{align}\label{equation:typeI:vcd:nondominant:prime}\nonumber
\Big\vert \sum_{\substack{m \sim \frac{X}{n}}} a_m e(yf(mn))\Big\vert^2&\ll \frac{M^2}{H}\log^3 M\\
&+ \frac{M}{H}\sum_{1<\vert h\vert \leq H} \left(1-\frac{\vert h\vert}{H+1}\right)\ \sum_{\substack{m\sim M\\ nm\sim X\\ (m+h)n\sim X}} a_m^\star a_{m+h}  e\left(yf_h(n,m)\right).
\end{align}
We now notice that since $H$ is small compared to $M$, we can suppress the condition in 
the innermost  $(m+h)n\sim X$ at the price of an (additive) error about $HM$ in equation (\ref{equation:typeI:vcd:nondominant:prime}), whence an error of size $XNH$ to $S^2$, which is negligible. Now we interchange summation and we obtain
\begin{align*}
S^2\ll_f \frac{X}{M}\times \left(\frac{MX}{H}\log^3 M+ \frac{M}{H}\sum_{1<\vert h\vert \leq H} \left(1-\frac{\vert h\vert}{H+1}\right)\ \sum_{\substack{m\sim M}} a_m^\star a_{m+h}\sum_{n\sim \frac{X}{M}} e\left(yf_h(n, m)\right)\right).
\end{align*}
Now, by partial summation we have
\begin{align*}
\sum_{n\sim \frac{X}{M}} e\left(yf_h(n, m)\right)&\ll \max_{Y\sim\frac{X}{M}} \Big\vert \sum_{\frac{X}{2M}\leq n\leq Y} e\left(yP_h(n, m)\right)\Big\vert \times \left(1+y \int_{\frac{X}{2M}}^{\frac{X}{M}} [\phi'(t(m+h))-\phi'(tm)]dt\right)\\
&\ll \max_{Y\sim\frac{X}{M}} \Big\vert \sum_{\frac{X}{2M}\leq n\leq Y} e\left(yP_h(n, m)\right)\Big\vert \times \left(1+X^{-\theta+\rho_\star}X^{\theta-1}X^\tau\frac{X}{M}\right)\\
&\ll \max_{Y\sim\frac{X}{M}} \Big\vert \sum_{\frac{X}{2M}\leq n\leq Y} e\left(yP_h(n, m)\right)\Big\vert,
\end{align*}
where in the last step we used $M\gg X^c$ and the definition of $c$ in \cref{equation:csplit:propintermfreq:type1:nondominant:prime}. In fact, if $M\gg X^{\frac{\theta-1}{k}}$ then by the definition of $\rho_\star$ we have $X^{-\theta+\rho_\star+\theta+\tau-\frac{\theta-1}{k}}\ll X^{-\epsilon}$ (keep in mind $\tau<\tau_1$). If otherwise we have $M\gg X^{\frac{1}{2}+\rho_\star}$ then we see, the integral is bounded by $X^{-\frac{1}{2}+\tau}\ll 1$.
\\\
\\\
Now, we treat $yP_h(n, m)$ as a full polynomial of degree $k$ (recall the differencing was on the $m$ variable) in $n$, having coefficients $ya_j[(m+h)^j-m^j]$. For the sequel we need the following claim, whose proof is given in \cref{Section:diophantine:claims}
\begin{Claim}\label{claim:proofpropintermediatefrequencies:claim2:type1:nondominant:prime}
Let $m\sim M$ and $M\gg X^c$, with $c$ defined in \cref{equation:csplit:propintermfreq:type1:nondominant:prime}. Let $X^{-\frac{23k}{30}}\ll y\ll X^{-\theta+\rho_\star}$. Then there is a pair $(a,q)$ of coprime integers and an index $2\leq j\leq k$ such that
\begin{align}
\Big\vert qya_j[(m+h)^j-m^j]- a\Big\vert\leq \frac{M^{j-1}}{X^{j-1}}
\end{align}
and
\begin{align}
\frac{X}{M}\leq q \leq \frac{X^{j-1}}{M^{j-1}}.
\end{align}
\end{Claim}
Now, applying the claim coupled with \cref{lemma:bourgainvinogradov:nondominant:prime} we get
\begin{align*}
\sum_{n\sim \frac{X}{M}} e\left(yP_h(n, m)\right)\ll_f \left(\frac{X}{M}\right)^{1-\frac{1}{k(k-1)}+\epsilon}.
\end{align*}
Hence, since $(a_m)_m$ is divisor bounded and $M\leq X^{3/5}$ we have
\begin{align*}
S^2&\ll_f \frac{X}{M}\times \left(\frac{MX}{H}\log^3 M+ \frac{M}{H}\sum_{1<\vert h\vert \leq H} \left(1-\frac{\vert h\vert}{H+1}\right)\ \sum_{\substack{m\sim M}} a_m^\star a_{m+h} \Big\vert \sum_{n\sim \frac{X}{M}} e\left(yf_h(n, m)\right)\Big\vert \right)\\
&\ll_f \frac{X^{2+\epsilon}}{H}+ X^{1-\frac{1}{k(k-1)}+\epsilon}X^{\frac{3}{5k(k-1)}}\\
&\ll \frac{X^{2+\epsilon}}{H}+ X^{2-\frac{2}{5k(k-1)}+\epsilon}.
\end{align*}
Now, specializing $\tau=\frac{2}{5k(k-1)}$ we obtain
\begin{align}\label{prop:intermediatefrequencies:type1:secondhalfbound:nondominant:prime}
S\ll_f X^{1-\frac{1}{5k(k-1)}+\epsilon}.
\end{align}
Comparing this with (\ref{prop:intermediatefrequencies:type1:firsthalfbound:nondominant:prime}) we are done.
\end{proof}

\section{Type II estimates}\label{section:typeII}
\begin{prop}\label{prop:type2expsums:nondominant:prime}
Let $f=P+\phi$ be a pseudo-polynomial with property (F), with $P$ of  degree $k$ and $\deg(\phi)=\theta$. Let $(a_m)_{m\in \mathbb{N}}$ and $(b_n)_{n\in\mathbb{N}}$ be divisor bounded sequences and $X^{1/5}\ll M\ll X^{\frac{1}{3}}$. Assume that
$$X^{-\frac{23k}{30}}\ll y\ll X^{1/2}.$$
Then
\begin{align}\label{Prop:largefrequenciessumtoestimate:typeII:nondominant:prime}
\sum_{m\sim M}\sum_{\substack{mn\sim X\\ n\sim N}} a_mb_n e(yf(mn))\ll_f  X^{1-\frac{1}{3k(k+1)^2}+\frac{1}{5k^2(k-1)(k+1)^2}+\epsilon}
\end{align}
\end{prop}

\begin{rem}
	The strategy of proof is similar to the one adopted in the preceding section, but perhaps even more transparent. Indeed, the use of the Weyl-Van der Corput inequality arises naturally. Similarly as before, we must distinguish among "high frequencies" and "small and intermediate frequencies". Depending on the size of $\deg(\phi)$, the range in \cref{prop:intermediatefrequencies:typeII:nondominant:prime}  can be already covered by the range in \cref{prop:highfrequencies:typeII:nondominant:prime}. The proposition above will follow directly by the following three propositions.
\end{rem}

	\begin{rem}
	As the careful reader will certainly spot, the condition on $B$ is required to ensure that \cref{HBlemma} will not return a trivial estimate because of $F\ll 1$. We will deal with the range $X^{-\theta}\ll y\ll X^{-\theta+B}$ in \cref{prop:type2:abitmorespace}. This additional subdivision is necessary here as a consequence of the application of the Weyl-Van der Corput differencing even in the "high frequencies" regime.
\end{rem}

\begin{prop}[High frequencies]\label{prop:highfrequencies:typeII:nondominant:prime}
Let $X^{-\theta+B}\ll y\ll X^{1/2}$, where 
\begin{align}
B:=\frac{2}{3}-\frac{1}{5k(k-1)}-\epsilon.
\end{align}
 Then we have
\begin{align*}
S\ll_f X^{1-\frac{2}{(k+1)(\frac{3}{2}k+\frac{5}{2})^2}+\epsilon}.
\end{align*}
\end{prop}

\begin{proof}
By Cauchy's inequality we have
	\begin{equation}\label{eq1pflemmaT2estimatelarge}
	\vert S\vert^2 \le \left(\sum_{n\sim N} \vert b_n\vert^2 \right)\times  \sum_{n\sim N} \Big\vert \sum_{\substack{m \sim M\\ nm \sim X}} a_m e(yf(mn))\Big\vert^2.
	\end{equation}
	Then, again by Van der Corput lemma with $H=X^\tau$ for some small $\tau$ to be specified later, we obtain
	\begin{align*}
	\Big\vert \sum_{\substack{m \sim M\\ nm \sim X}} a_m e(yf(mn))\Big\vert^2&\ll \frac{M^2}{H}\log^3 M \\
	&+ \frac{M}{H}\sum_{1<\vert h\vert \leq H} \left(1-\frac{\vert h\vert}{H+1}\right)\ \sum_{\substack{m\sim M\\ nm\sim X\\ (m+h)n\sim X}} a_m^\star a_{m+h} e\left(yf_h(n, m))\right).
	\end{align*}
	 Notice once again that we can remove the condition $(m+h)n\sim X$ at the price of an additive error for $S^2$ of size $N^2MH$. Interchanging summation we arrive at 
	\begin{align}\label{eq:typeII:CauchyVdC:nondominant:prime}
	S^2 &\ll \left(\sum_{n\sim N} \vert b_n\vert^2\right)\\ \nonumber\times \Biggl(&\frac{M^2N}{H}\log^3 M+ \frac{M^{1+\epsilon}}{H}\sum_{1\leq \vert h\vert \leq H} \left(1-\frac{\vert h\vert}{H+1}\right) \sum_{m\sim M} a_m^\star a_{m+h}\Big\vert \sum_{\substack{n\sim N\\ mn\sim X}} e\left(yf_h(n, m)\right) \Big\vert\Biggr).
	\end{align}
	We proceed now by estimating the innermost sum. To this end we set $yX^\theta=X^\alpha$ and $$j:=\max\left(k+1,\lceil \frac{3}{2}(\alpha+\tau)\rceil +1 \right).$$
	We see that 
	\begin{align}
	yhM^{\theta-1}N^{\theta-j}\ll_f \partial_n^{j} yf_{h}(n, m)\ll_f yh M^{\theta-1}N^{\theta-j}.
	\end{align}
hence, applying \cref{HBlemma} with differentiation level $j$ and $F=\vert h\vert \frac{X^\alpha}{M}$ and noticing that $FN^{-j}\ll N^{-1}$ in virtue of our selection of $j$, we obtain the bound
\begin{align*}
\sum_{n\sim N} e\left(yf_h(n, m)\right) &\ll_f N^{1+\epsilon}\times \left( \left(FN^{-j}\right)^{\frac{1}{j(j-1)}}+N^{-\frac{1}{j(j-1)}}+F^{-\frac{2}{j^2(j-1)}}\right)\\
&\ll_f N^{1+\epsilon}\times  \left(N^{-\frac{1}{j(j-1)}}+F^{-\frac{2}{j^2(j-1)}}\right)\\
&\ll_f N^{1-\frac{1}{j(j-1)}+\epsilon} + N^{1+\epsilon}\left( \vert h\vert \frac{X^\alpha}{M}\right)^{-\frac{2}{j^2(j-1)}}.
\end{align*}
Thus, as the sequence $(a_m)_m$ is divisor bounded, summing over $m\sim M$ in (\ref{eq:typeII:CauchyVdC:nondominant:prime}) we obtain 
\begin{align*}\label{prop:highfrequencies:finalequation:nondominant:prime:typeII}\nonumber
S^2& \ll \left(\sum_{n\sim N} \vert b_n\vert^2\right)\times \left[\frac{M^2N}{H}\log^3 M +\frac{M^{1+\epsilon}}{H}\sum_{1\leq h\leq H} \left(  MN^{1-\frac{1}{j(j-1)}}+X^{1+\epsilon} \left( \vert h\vert \frac{X^\alpha}{M}\right)^{-\frac{2}{j^2(j-1)}}\right)\right]\\
&\ll_f \frac{X^{2+\epsilon}}{H}+X^{1+\epsilon} MN^{1-\frac{1}{j(j-1)}}+X^{2+\epsilon} H^{-\frac{2}{j^2(j-1)}} X^{-\frac{2\alpha}{j^2(j-1)}} M^{\frac{2}{j^2(j-1)}}\\
&\ll_f \frac{X^{2+\epsilon}}{H}+X^{2-\frac{2}{3j(j-1)}+\epsilon}+X^{2-\frac{2(\alpha-\frac{1}{3})}{j^2(j-1)}+\epsilon}.
\end{align*}
Let us consider first the case when $j=k+1$ (notice that this eventuality takes actually place when $y$ is small). Then, selecting $\tau=\frac{2}{3k(k+1)}$ and using $\alpha\geq B$ we are left with the bound
\begin{align*}
S\ll_f X^{1-\frac{1}{3k(k+1)}+\epsilon}+ X^{1-\frac{1}{3k(k+1)^2}+\frac{1}{5k^2(k-1)(k+1)^2}+\epsilon}.
\end{align*}
If otherwise $j>k+1$ then $\alpha>6$, in this case we select $\tau=\frac{2}{3j(j-1)}$ and we compute
\begin{align*}
\frac{2(\alpha-\frac{1}{3})}{j^2(j-1)}\geq \frac{10}{\frac{9}{4}\left(\alpha+\tau\right)\left(\frac{3}{2}(\alpha+\tau)+1\right)^2}\geq \frac{4}{\left(\alpha+\tau\right)\left(\frac{3}{2}(\alpha+\tau)+1\right)^2}.
\end{align*}
The latter is a decreasing function in $\alpha$ provided $\alpha>1$. Hence, using $\alpha+\tau<k+1$ and selecting $\tau=\frac{2}{3j(j-1)}$ we arrive at
\begin{align}
S\ll X^{1-\frac{2}{(k+1)(\frac{3}{2}k+\frac{5}{2})^2}+\epsilon}.
\end{align}
Performing a quick comparison of this bound with the one obtained for the $j=k+1$ case gives the claim.
\end{proof}

\begin{prop}\label{prop:type2:abitmorespace}
	Let $X^{-\theta}\ll y\ll X^{-\theta+B}$. Then we have\begin{align*}
	S\ll X^{1-\frac{1}{10k(k-1)}+\epsilon}.
	\end{align*}
\end{prop}

\begin{proof}
	By Cauchy's inequality we have 
	\begin{align*}
	\vert S\vert^2 \le \left(\sum_{m\sim M} \vert a_m\vert^2 \right)\times  \sum_{m \sim M} \Big\vert \sum_{\substack{n \sim N\\ nm \sim X}} b_n e(yf(mn))\Big\vert^2.
	\end{align*}
	Then, again by Van der Corput lemma with $H=X^\tau$ for some small $\tau$ to be specified later and proceeding as in the previous proposition we arrive at
	\begin{align*}\label{eq:typeII:CauchyVdC:nondominant:prime}
	S^2 &\ll \frac{X^{2+\epsilon}}{H}+\frac{X^{1+\epsilon}}{H}\sum_{1\leq \vert h\vert \leq H} \left(1-\frac{\vert h\vert}{H+1}\right) \sum_{n\sim N} b_n^\star b_{n+h}\Big\vert \sum_{\substack{m \sim M\\ mn\sim X}} e\left(yf_h(m, n)\right) \Big\vert,
	\end{align*}
	where differencing took place on the variable $n$. By partial summation we obtain
	\begin{align*}
	\sum_{\substack{m\sim M}} e\left(yf_h(m, n)\right)&\ll_f \max_{X/2N\leq Y} \Big\vert \sum_{\substack{m \sim M\\  m\leq Y}} e\left(yP_h(m, n)\right)\Big\vert \times \left(1+y\int_{\frac{X}{2N}}^{\frac{X}{N}} \left(\phi'(t(n+h))-\phi'(nt)\right)dt\right)\\
	&\ll_f \sum_{m \sim M} e\left(yP_h(m, n)\right),
	\end{align*}
	since the integral is bounded by
	\begin{align*}
	\ll X^{-\theta+B}\frac{X}{N} N^{\theta-1}M^{\theta-1}H=\frac{X^BH}{N}\ll 1,
	\end{align*}
	as a consequence of the selection we will take for $\tau$ and $N\gg X^{\frac{2}{3}}$.
	Now, notice that $yP_h(m, n)$ is a full polynomial in $m$ with coefficients given by $ya_j\left((n+h)^j-n^j\right)$. 
	The following claim, whose proof is postponed to the next section, shows that we can appropriately approximate some of these coefficients to meet the conditions for applying \cref{lemma:bourgainvinogradov:nondominant:prime}.

	\begin{Claim}\label{claim3}
		Let $X^{-\theta}\ll y\ll X^{-\theta+B}$. Then there's some $2\leq j\leq k$ and some pair of coprime integers $(a,q)$ such that 
		\begin{equation}\label{eq1:claim3}
		\Big\vert qya_j\left((n+h)^j-n^j\right)-a\Big\vert \leq M^{1-j}
		\end{equation}
		and
		\begin{align}
		M\ll q\ll M^{j-1}.
		\end{align}
	\end{Claim}
Applying now \cref{lemma:bourgainvinogradov:nondominant:prime}	
	\begin{align*}
	\sum_{m\sim M} e\left(yf_h(m, n)\right) \ll M^{1-\frac{1}{k(k-1)}+\epsilon}
	\end{align*}
	Hence, taking into account $M\gg X^{\frac{1}{5}}$, we have
	\begin{align*}
	S^2\ll \frac{X^{2+\epsilon}}{H}+ X^{2-\frac{1}{5 k(k-1)}+\epsilon},
	\end{align*}
	whence selecting $\tau=\frac{1}{5k(k-1)}-\epsilon$ we are done.
\end{proof}

\begin{prop}[Intermediate and low frequencies]\label{prop:intermediatefrequencies:typeII:nondominant:prime}
Let $X^{-\frac{23k}{30}}\ll y\ll X^{-\theta}$. Then 
\begin{align}
S\ll_f X^{1-\frac{1}{6k(k-1)}+\epsilon}.
\end{align}
\end{prop}
\begin{proof}[Proof of \cref{prop:intermediatefrequencies:typeII:nondominant:prime}]
Proceeding as in the proof of \cref{prop:highfrequencies:typeII:nondominant:prime} we arrive at \cref{eq:typeII:CauchyVdC:nondominant:prime}. Now we remove the the "noise" $y\phi(mn)$ in the innermost sum in \cref{eq:typeII:CauchyVdC:nondominant:prime}. By partial summation:
\begin{align*}
\sum_{\substack{n\sim N}} e\left(yf_h(n,m)\right)&\ll_f \max_{X_m/2\leq Y} \Big\vert \sum_{\substack{n\sim N\\ nm\sim N\\ n\leq Y}} e\left(yP_h(n, m)\right)\Big\vert \times \left(1+y\int_{\frac{X_m}{2}}^{X_m} \left(\phi'(t(m+h))-\phi'(mt)\right)dt\right).
\end{align*}
The integral above is bounded by  $\ll X^{-\theta}X_m M^{\theta-1}N^{\theta-1}H\ll X^{-\theta+\tau}X^{\theta-1}\frac{X}{M}\ll X^{\tau}M^{-1}\ll 1$, where the last inequality will follow since we will pick $\tau$ to be appropriately small and $M\gg X^{1/5}$. Hence
\begin{align*}
\sum_{\substack{n\sim N\\ nm\sim X}} e\left(yf_h(n, m)\right) \ll \max_{X_m/2\leq Y} \Big\vert \sum_{\substack{n\sim N\\ nm\sim N\\ n\leq Y}} e\left(yP_h(n, m)\right)\Big\vert.
\end{align*}
Therefore, we only need to estimate the Weyl sum 
\begin{align*}
\sum_{\substack{n\sim N}} e\left(yP_h(n, m)\right).
\end{align*}
We treat $yP_h(n, m)$ as a (full!) polynomial of degree $k$ in the variable $n$ having coefficients  $y\alpha_j\left((m+h)^j-m^j\right)$. In the \cref{claim:diophantineproperty:typeII:nondominant:prime} we prove that for $y$ in our range, there is always one coefficient (actually depending on the size of $y$) of the polynomial $yP_h(n, m)$ (seen as a polynomial in $n$) with a rational approximation which is satisfactory towards an application of \cref{lemma:bourgainvinogradov:nondominant:prime}.
\begin{Claim}\label{claim:diophantineproperty:typeII:nondominant:prime}
Let $m\sim M$ and let  $X^{-\frac{23k}{30}}\ll y\ll X^{-\theta}$. Then, there is an index $2\leq j\leq k$ and a pair $(a,q)$ of coprime integers such that 
\begin{align}\label{eq:gooddioph:cond1:typeII:nondominant:prime}
\Big\vert qy a_j\left((m+h)^j-m^j\right)-a\Big\vert \leq \frac{m^{j-1}}{X^{j-1}},
\end{align}
and
\begin{align}\label{eq:gooddioph:cond2:typeII:nondominant:prime}
X_m\ll q\ll X_m^{j-1}.
\end{align}
\end{Claim}
Applying the claim together with \cref{lemma:bourgainvinogradov:nondominant:prime} we obtain
\begin{align*}
\sum_{\substack{n\sim N}} e\left(yP_h(n, m)\right)&\ll_f X_m^{1-\frac{1}{k(k-1)}+\epsilon}.
\end{align*}
Inserting this into \cref{eq:typeII:CauchyVdC:nondominant:prime} we have
\begin{align*}
S^2&\ll_f N^{1+\epsilon}\times \left(\frac{M^2N}{H}\log^3M+M^{1+\epsilon}\sum_{m\sim M} X_m^{1-\frac{1}{k(k-1)}+\epsilon}\right)\\
&\ll_f \frac{X^{2+\epsilon}}{H}+X^{2-\frac{2}{3k(k-1)}+\epsilon}.
\end{align*}
Finally, selecting again $\tau=\frac{2}{3k(k-1)}$ we obtain
\begin{align*}
S\ll_f X^{1-\frac{1}{3k(k-1)}+\epsilon}.
\end{align*}
\end{proof}

\section{Diophantine approximations}\label{Section:diophantine:claims}
In this final section we give a proof of the four claims used in the preceeding two sections. The strategy will be the same for all four lemmas. However, for the benefit of the reader, we decided to give fully detailed proofs of all four.
\begin{proof}[\textbf{Proof of \cref{Claim:Dioph}}]
We recall that we work under the assumption $M\ll X^c$, with $c$ defined as in \cref{equation:csplit:propintermfreq:type1:nondominant:prime}. Moreover, we recall that we will eventually pick the parameter $\tau$ such that $\tau<\tau_1$.
The existence of the pair $(a,q)$ satisfying inequality (\ref{Dioph:prop:1}) is granted by Dirichlet's theorem:
\begin{align}\label{Claim:dioph:contr:nondominant:prime}
\vert q ya_k(km^kh)-a\vert \leq \left(\frac{m}{X}\right)^{k-2},
\end{align}
with $1\leq q\leq X_m^{k-2}$. Hence, we only need to check that the denominator is of the required size $q\gg X_m$. We consider separately the cases $q\not =1$ and $q=1$.
\begin{itemize}
	\item Case $q>1$. In this case $a\not=0$ and an application of the triangle inequality gives us 
\begin{align*}
\vert ya_k(m^k h)\vert \geq \Big\vert\frac{a}{q}\Big\vert-\frac{1}{q^2}\gg \frac{1}{q}.
\end{align*}
Now, since $c\leq \frac{\theta-1}{k}$ then the left side of the above inequality is upper bounded by 
$$yhM^k=X^{-\theta+\rho_\star}X^{\tau}X^{\theta-1}\ll X^{\rho_\star+\tau-1},$$
which forces $q\gg X^{1-\rho_\star-\tau}$. On the other hand, since $m\sim M$ we have $X/m\gg X^{1-c}$. Thus $q\gg X^{1-\rho_\star-\tau}\gg X^{1-c}\gg \frac{X}{m}$, as a consequence of our selection of $\rho_\star$ and $\tau$.
\item Case $q=1$ and $a\not =0$. From the previous point we see that $ya_km^kh=o(1)$. So this eventuality can be ruled out because of \cref{Claim:dioph:contr:nondominant:prime}.
\item Case $q=1$ and $a=0$. This eventuality can be ruled out easily. Indeed, in this case equation \cref{Claim:dioph:contr:nondominant:prime} reads $\vert ya_k (km^k h)\vert \leq \frac{m^{k-2}}{X^{k-2}}\Leftrightarrow \vert 
ka_kym^2h\vert \leq X^{2-k},$ which is impossible as $y\gg X^{-\frac{23k}{30}}$ and $k> 11$.
\end{itemize}
This completes the proof.
\end{proof}

\begin{proof}[\textbf{Proof of \cref{claim:proofpropintermediatefrequencies:claim2:type1:nondominant:prime}}]
The proof follows the same lines as the one we gave for \cref{Claim:Dioph}.  Write $y=X^\beta$ and set $j:=\lceil -\beta+ 2\tau\rceil+1$ 
(we recall, $\tau$ was the parameter in the application of the Van der Corput inequality). Notice that since $\tau$ is small and and $-\frac{23k}{30}\leq \beta\leq -\theta+\rho_\star\leq -3$, we have $2\leq j\leq k$. Also recall that since we have $M\gg X^c$ and $\tau<c$, we have
$$M^{j-1}h\ll (m+h)^j-m^j\ll M^{j-1}h.$$
Now, an application of Dirichlet's theorem gives us the first assertion:
\begin{align}\label{eq:proof:claim:diophII}
\Big\vert q(ya_j[(m+h)^j-m^j])- a\Big\vert\leq \frac{M^{j-1}}{X^{j-1}}.
\end{align}
We must now check that the denominator $q$ fits into the required range.
\begin{itemize}
	\item Case $q>1$. In this case, by triangle inequality we have
\begin{align*}
\vert y a_j [(m+h)^j-m^j]\vert\geq \Big\vert\frac{a}{q}\Big\vert-\frac{1}{q^2}\gg \frac{1}{q}.
\end{align*}
The left hand side is $\gg X^\beta M^{j-1}h$. Now, if the denominator $q$ 
is at most $\frac{X}{M}$, then, as the left hand side above is $X^\beta M^{j-1}h$, estimating $h\ll X^\tau$ we would have 
\begin{align*}
X^{\beta+\tau+1}M^{j-2}\gg 1.
\end{align*}
However, as $M\ll X^{3/5}$, the left hand side above is bounded by
\begin{align*}
X^{\beta+\tau+1}X^{\frac{3(j-2)}{5}}\ll X^{\frac{2}{5} \beta+\frac{11}{5}\tau+1}
\end{align*}
where we used our selection of $j$. However, we see that 
\begin{align*}
\frac{2}{5}\beta+\frac{11}{5}\tau+1<0
\end{align*}
since $\beta\leq -3$ and $\tau\leq \frac{1}{20}$. This gives a contradiction, so $q\gg X/M$.
\item Case $q=1$ and $a\not =0$. We can discard this eventuality because otherwise \cref{eq:proof:claim:diophII} reads
\begin{align*}
\vert ya_j[(m+h)^j-m^j]-a\vert\leq \left(\frac{M}{X}\right)^{j-1}, 
\end{align*}
which is absurd since $ya_j[(m+h)^j-m^j]\ll yM^{j-1}X^{\tau}=o(1)$ (to see this, keep in mind $\tau_1<1/20$ and $M\ll X^{\frac{3}{5}}$, as well as the definition of $j$).
\item Finally, if $q=1, a=0$ then \cref{eq:proof:claim:diophII} reads 
\begin{align*}
\vert ya_j[(m+h)^j-m^j]\vert\leq \left(\frac{M}{X}\right)^{j-1}.
\end{align*} 
Rearranging, we arrive at the inequality $X^{\beta}hX^{j-1}\ll 1$. This is impossible as the left side is $\gg X^{\beta -\beta+2\tau}\gg X^{2\tau}$ by our selections of $j$ and the upper bound for $\beta$.
\end{itemize}
This concludes the proof of the claim.
\end{proof}

\begin{proof}[\textbf{Proof of \cref{claim3}}]
	Set $y=X^\beta$ and $j:=\lceil -\beta\rceil+1$. We see that $-k<\beta<-3$ thus $2\leq j\leq k$. Moreover, since $N\gg X^{\frac{2}{3}}$, $n\sim N$ and since we will select $\tau<2/3$ we have
	\begin{align*}
	N^{j-1}h\ll (n+h)^j-n^j\ll N^{j-1}h.
	\end{align*}
	Now, the first assertion is obtained again by Dirichlet's theorem. For the second we must distinguish again among three possible cases:
	\begin{itemize}
		\item Case 1: $q>1$.
		In this case, by triangle inequality we obtain
		\begin{align*}
		yN^{j-1}H\gg \frac{1}{q}.
		\end{align*}
		The latter is impossible for $q\ll M$.

		\item Case 2: $q=1$ and $a\not=0$. This case can be ruled out since
		\begin{align*}
		yN^{j-1}H\ll X^{\beta}X^{4/5\lceil -\beta \rceil+\tau}=o(1).
		\end{align*}		
		\item Case 3: $q=1$ and $a=0$. In this case \cref{eq1:claim3} reads
		\begin{align*}
		yN^{j-1}H\ll N^{1-j}, 
		\end{align*}
		which is impossible: indeed, since $N\gg X^{2/3}$ we are led to
		\begin{align*}
		X^{\beta}X^{\frac{4}{3}\lceil -\beta\rceil+\tau}\ll 1,
		\end{align*}
		which is impossible.
	\end{itemize}
\end{proof}

\begin{proof}[\textbf{Proof of \cref{claim:diophantineproperty:typeII:nondominant:prime}}]
Let us write again $y=X^\beta$. Set $j:=\lceil -\beta+\tau\rceil +1$. Once again notice that $2\leq j\leq k$. By Dirichlet's theorem we obtain the first assertion. To prove that the denominator $q$ satisfies (\ref{eq:gooddioph:cond2:typeII:nondominant:prime}), we distinguish among the following cases.
\begin{itemize}
\item If $q>1$ then by triangle inequality we have
\begin{align*}
y\alpha_j\left((m+h)^j-m^j\right)\gg \frac{1}{q}.
\end{align*}
But the latter is impossible for $q\leq X/M$. To see this, notice that since $N\ll X^{4/5}$ and $M\ll X^{\frac{1}{3}}$ we have
\begin{align*}
yM^{j-1}H\ll B^{\beta+\frac{1}{3}\left(\lceil -\beta+\tau\rceil\right)+\tau}.
\end{align*}
However, we have
\begin{align*}
\beta+\frac{1}{3}\left(\lceil -\beta+\tau\rceil\right)+\tau<-\frac{4}{5}.
\end{align*}
Thus, $q\gg N$ and we are done.
\item Case $q=1$ and $a\not =0$. This can be excluded since by $yM^{j-1}h\ll  X^{\beta+\tau}X^{\frac{1}{3}(-\beta+\tau+1)}=o(1)$.
\item 
Case $q=1$ and $a=0$. In this eventuality we would have 
\begin{align*}
\Big\vert y\alpha_j\left((m+h)^j-m^j\right)\Big\vert \leq \frac{M^{j-1}}{X^{j-1}},
\end{align*}
which is impossible. Indeed, as the left hand side is $\gg yM^{j-1}$, rearranging we get $y\ll X^{1-j}$. The latter is impossible, since $\beta>-\lceil -\beta +\tau\rceil$.

\end{itemize}
\end{proof}

\bibliographystyle{plainnat}
\bibliography{universalbib}
\end{document}